\documentclass[11pt]{article}

\usepackage{amsmath}

\usepackage{breqn} 

\usepackage[hyperfootnotes=false]{hyperref} 
\usepackage{xcolor}
\definecolor{Bath}{RGB}{0,59,209}
\hypersetup{
colorlinks=true,
linkcolor=Bath,
citecolor=Bath,
urlcolor=Bath
}

\usepackage{amsmath,amsfonts,amssymb,amsthm,epsfig,epstopdf,url,array}
\usepackage{bm}
\usepackage{float}
\usepackage{mathtools}
\usepackage{esint}

\usepackage[nottoc,notlot,notlof]{tocbibind}

\usepackage{enumitem}

\usepackage[toc,page]{appendix}

\usepackage{graphicx}
\makeatletter
\newcommand*\bigcdot{\mathpalette\bigcdot@{.5}}
\newcommand*\bigcdot@[2]{\mathbin{\vcenter{\hbox{\scalebox{#2}{$\m@th#1\bullet$}}}}}
\makeatother

\usepackage[a4paper,bindingoffset=0.2in,%
            left=1in,right=1in,top=1.5in,bottom=1.5in,%
            footskip=.25in]{geometry}

\newcommand{\norm}[1]{\left\lVert#1\right\rVert}
\newcommand{\normm}[1]{\lVert#1\rVert}
\newcommand{\hvarphi}{\hat{\varphi}}

\newcommand{\tgamma}{\tilde{\gamma}}

\newcommand{\vecpar}[1]{\frac{\partial}{\partial x^{#1}}}
\DeclareMathOperator*{\esssup}{ess\,sup}

\newcommand{\hiip}[2]{\langle #1 , #2 \rangle}

\usepackage[symbol]{footmisc}

\newtheorem{thm}{Theorem}

\newtheorem{lem}[thm]{Lemma}
\newtheorem{rem}[thm]{Remark}

\newtheorem{defn}[thm]{Definition}
\newtheorem{prop}[thm]{Proposition}

\providecommand{\keywords}[1]
{
  \small	
  \textbf{\textit{Keywords---}} #1
}

\providecommand{\msc}[1]
{
  \small	
  \textbf{\textit{MSC---}} #1
}

\usepackage{setspace}
\begin{document}

\title{\vspace{-3cm}The $\infty$-elastica problem on a Riemannian manifold\vspace{-0.2cm}}
\author{
Ed Gallagher \thanks{Department of Mathematical Sciences, University of Bath, Bath, BA2 7AY, UK. Email: \href{redg22@bath.ac.uk}{redg22@bath.ac.uk}}
\and
Roger Moser \thanks{Department of Mathematical Sciences, University of Bath, Bath, BA2 7AY, UK. Email: \href{rm257@bath.ac.uk}{rm257@bath.ac.uk}}
}
\date{ \vspace{-1cm}}

\maketitle


\begin{abstract}
\noindent We consider the following problem: on any given complete Riemannian manifold $(M,g)$, among all curves which have fixed length as well as fixed end-points and tangents at the end-points, minimise the $L^\infty$ norm of the curvature. We show that the solutions of this problem, as well as a wider class of curves, must satisfy a second order ODE system. From this system we obtain some geometric information about the behaviour of the curves.
\end{abstract}

\keywords{elastica, $L^\infty$ variational problem, curvature, Riemannian manifold}

\msc{49Q20, 53A04, 58E10}

\section{Introduction}\label{sec:intro}

\noindent The study of so-called ``variational problems in $L^\infty$'' originated in the 1960s in a series of papers by Aronsson \cite{aronsson1965minimization} \cite{aronsson1966minimization} \cite{aronsson1969minimization} with a focus on first order functionals; these have now been studied extensively. In contrast, second order functionals have a significantly shorter history, having started to be investigated only relatively recently. In the one-dimensional case, problems of second order and higher (with no immediate geometric context) are touched on in \cite{aronsson2010certain} and \cite{aronsson2012variational}. In \cite{katzourakis2019existence}, a variational problem depending on second derivatives only through the Laplacian is studied, and in \cite{katzourakis2020second} general second-order $L^\infty$ variational problems concerning the minimisation of an arbitrary $C^1$ function of the Hessian $D^2 u$ are examined. \\
\noindent Specific $L^\infty$ variational problems arising from geometric motivations have been studied for longer than the generic second order problems, and some of these can be considered second order in the sense that (say) the quantities of interest require second derivatives when expressed in local co-ordinates. Some examples include \cite{moser2012minimizers}, where surfaces minimising a weighted $L^\infty$ norm of the Gauss curvature are considered, and \cite{sakellaris2017minimization} where minimisers among a conformal class of the $L^\infty$ norm of the scalar curvature are studied for manifolds of dimension at least three. \\

\noindent In this paper, we consider the problem of minimising the $L^\infty$ norm of the curvature of a curve subject to boundary conditions and a side constraint. Specifically, we consider the problem with the following setup: let $(M,g)$ be a complete Riemannian manifold, and fix two points $x_1, x_2 \in M$ as well as two unit tangent vectors $v_1 \in T_{x_1}M$, $v_2 \in T_{x_2}M$ at these points. Fix also a length $L \geq d(x_1,x_2)$, where $d(\bigcdot \, , \, \bigcdot)$ denotes the Riemannian distance function on $M$. We identify curves in $M$ with their arclength parametrisation $\gamma: [0,L] \rightarrow M$ with $\vert \gamma' \vert \, = \vert T \vert \, \equiv 1$, and the (unsigned) curvature is given by $\kappa = \vert \nabla_T T \vert$. Note that the arclength condition affects the problem only superficially since curvature is invariant under reparametrisation. Assuming that $\gamma$ belongs to the Sobolev space $W^{2,\infty}((0,L); M)$-- the natural space to work in for this problem, which we will denote by the shorthand $W^{2,\infty}(0,L)$ or $W^{2,\infty}$ in places where the meaning is clear-- we introduce the functional
\[
\mathcal{K}_\infty \left[ \gamma \right] = \esssup_{s \in [0,L]} \, \vert \nabla_T T \vert.
\]
Denoting by $\mathcal{G}$ the set of all curves in $W^{2,\infty}(0,L)$ which satisfy the boundary conditions
\begin{equation}\label{eq:bcs}
\gamma(0) = x_1, \quad \gamma(L) = x_2, \quad \gamma'(0) = v_1, \quad \gamma'(L) = v_2
\end{equation}
and the length constraint
\begin{equation}\label{eq:length}
\mathcal{L} \left[ \gamma \right] = L,
\end{equation}
as well as the arclength condition $\vert \gamma' \vert \, \equiv 1$ (that is, $\mathcal{G}$ is the set of admissible curves for the problem), we seek to understand the behaviour of minimisers of $\mathcal{K}_\infty$ in $\mathcal{G}$. As we will see later on, such minimisers always exist and so the problem is always meaningful. \\

\noindent Note that here we are defining the space $W^{2,\infty}((0,L); M)$ as the space of all functions from $(0,L)$ to $M$ whose expressions in suitable co-ordinate charts belong to the Euclidean Sobolev space $W^{2,\infty}$, with similar definitions for the spaces $W^{2,q}$ with $1 \leq q < \infty$. \\

\noindent This problem has not been considered in full generality previously, though specific cases of it and problems of a similar nature have been investigated. The problem in $L^p$ has been considered before, e.g. in \cite{huang2004note} which focuses on manifolds of constant sectional curvature, and it has been studied extensively in the case $p=2$ where the theory originated with the Bernoulli brothers as well as Euler \cite{oldfather1933leonhard}. However, the change from the finite $p < \infty$ to the infinite $p = \infty$ gives rise to a fundamentally different problem for which none of the usual tools of the calculus of variations apply. \\
\noindent In Euclidean space $\mathbb{R}^n$, the $\infty$-elastica problem has already been studied by Moser \cite{moser2019structure}. The main novelty in the present paper, therefore, is the generalisation to arbitrary complete manifolds. Somewhat surprisingly, given that the $L^\infty$ norm is not differentiable in any material sense of the word, \cite{moser2019structure} shows that solutions of the Euclidean problem (as well as a more general class of curves) are characterised by a system of first order ODEs, derived from the limit as $p \rightarrow \infty$ of the Euler-Lagrange equations for the $L^p$ problem. Moreover, the solutions admit a classification based on their structure, being either a chain of two-dimensional curves or a single three-dimensional curve. \\

\noindent As well as minimisers of $\mathcal{K}_\infty$, the analysis in this paper also lets us draw conclusions about a weaker class of ``pseudo-minimisers'', called $\infty$-elastica, which are a sort of analogue of critical points for the $L^\infty$ norm. Due to the definition of the $L^\infty$ norm there is no meaningful way to define critical points in the traditional sense. The underlying idea is that as well as minimisers of $\mathcal{K}_\infty$, the definition of an $\infty$-elastica also captures `weaker' curves which minimise a modified $\mathcal{K}_\infty$ functional:
\begin{defn}\label{defn:inftyelastica}
A curve $\gamma \in \mathcal{G}$ is called an $\infty$-elastica if there exists a constant $M \in \mathbb{R}$ such that the inequality
\begin{equation}\label{eq:inftyelasdefn}
\mathcal{K}_\infty \left[\gamma\right] \leq \mathcal{K}_\infty \left[\tgamma\right]  + \frac{M}{2L} \int_0^{L} d\left(\tgamma(s), \gamma([0,L]) \right)^2 \, \mathrm{d}s  
\end{equation}
is satisfied for all $\tgamma \in \mathcal{G}$. \\
Here, the function $d$ denotes the (geodesic) distance between the point $\tgamma(s)$ and the set $\gamma([0,L])$, i.e. the image of the curve $\gamma$.
\end{defn}
\noindent Although we use the same language (``$\infty$-elastica'') as \cite{moser2019structure}, our definition is in fact slightly different (cf. \cite[Definition~1]{moser2019structure}), although both definitions capture the minimisers of $\mathcal{K}_\infty$. \\

\noindent The purpose of this paper is to prove the following theorem, which extends some of the theory from \cite{moser2019structure} from Euclidean space $\mathbb{R}^n$ to an arbitrary complete Riemannian manifold $(M,g)$:
\begin{thm}\label{thm:limitingeqns}
Assume that $\gamma \in \mathcal{G}$ is an $\infty$-elastica and that $\gamma$ is not a geodesic. Let $K = \mathcal{K}_\infty \left[ \gamma \right]$. Then the following statements hold:\\
1) There exist a vector field $\varphi \in W^{2,\infty}_{\text{loc}}(0,L) \backslash \{ 0 \}$ defined along $\gamma$ and a number $\lambda \in \mathbb{R}$ such that the equations 
\begin{align} 
\nabla_{T}^2 \varphi + R\left(\varphi,T\right)T &= L \lambda \nabla_{T} T - 2\nabla_T \left( \left\langle \varphi, \nabla_{T} T \right\rangle T \right), \label{eq:EL1} \\
\vert \varphi \vert \nabla_T T &= K \varphi, \label{eq:EL2}
\end{align}
are satisfied a.e. in $(0,L)$. \\
2) The curvature of $\gamma$ takes on only two values up to null sets: $\kappa = K$ a.e. on the set where $\varphi \neq 0$, and $\kappa = 0$ a.e. on the set where $\varphi = 0$. \\
In particular, 1) and 2) hold when $\gamma$ is a minimiser of $\mathcal{K}_\infty$.
\end{thm}
\noindent Note that $\kappa$ denotes the \textit{unsigned} curvature, so it is possible that the normal vector may suddenly reverse direction in places. We therefore cannot expect any higher regularity that an $\infty$-elastica will possess any higher regularity than $W^{2,\infty}$ in general. The classification results of \cite{moser2012minimizers} demonstrate this: for curves in the plane, we can take the boundary and length conditions to be such that any singular circular arc is not admissible for the problem. Then a minimiser of $\mathcal{K}_\infty$-- which is an $\infty$-elastica in both the sense of this paper and the sense of \cite{moser2019structure}-- must have its normal vector suddenly reverse direction somewhere by \cite[Theorem~4]{moser2019structure}, i.e. its second derivative must be discontinuous. \\

\noindent Theorem \ref{thm:limitingeqns} opens up the possibility of further analysis of $\infty$-elastica through the ODE system \eqref{eq:EL1}--\eqref{eq:EL2} in a manner similar to that of \cite{moser2019structure}; however, such an analysis would most likely require us to restrict our attention to a single manifold $M$ at a time, and would potentially only work for more `well-behaved' and well-understood manifolds (e.g. the sphere $\mathbb{S}^n$ or hyperbolic space $\mathbb{H}^n$). Since the focus of this paper is on the general case, we do not consider this analysis here.

\noindent In the context of Theorem \ref{thm:limitingeqns}, and throughout this paper, the condition $\varphi \in W^{2,\infty}_{\text{loc}}(0,L)$ means that $\varphi \in W^{2,\infty}([a,b])$ for any $[a,b] \subset (0,L)$, where the space $W^{2,\infty}([a,b])$ may be interpreted as the space of all vector fields whose local co-ordinate representations as functions to $\mathbb{R}^n$ are in $W^{2,\infty}$ in suitable co-ordinate charts. It will be clear from context whether we are talking about a function or a vector field being in their respective Sobolev spaces. The restriction that $\gamma$ is not a geodesic does not take anything away from the problem, since if the conditions \eqref{eq:bcs}--\eqref{eq:length} allow geodesics as admissible curves then it is immediate that the \textit{only} minimisers of $\mathcal{K}_\infty$ will be such geodesics.\\

\noindent Intriguingly, the results of Theorem \ref{thm:limitingeqns} are different in some regards to the results for the problem in Euclidean space from \cite{moser2019structure}. Most notably, the ODE system \eqref{eq:EL1}--\eqref{eq:EL2} is of second order, compared to the first order system in \cite[Theorem~2]{moser2019structure}. This difference appears since in Euclidean space it is possible to consider everything in terms of the unit tangent vector field rather than the curve itself-- indeed, the curve can be recovered from only the initial data and tangent field via integration-- while such an approach is not possible on an arbitrary manifold. Moreover, the connection between the ODE system and $\infty$-elastica in \cite{moser2019structure} is an equivalence, whereas here the ODE system \eqref{eq:EL1}--\eqref{eq:EL2} has only been shown to be a necessary condition for $\infty$-elastica; this is a consequence of the specific way we have defined $\infty$-elastica, as the proof equivalence in \cite{moser2019structure} leans heavily on the (different) definition of $\infty$-elastica there, although the question of whether we in fact have equivalence in Theorem \ref{thm:limitingeqns} is unresolved. \\


\noindent The paper is structured as follows. In Section \ref{sec:approx}, we introduce the penalisation term to our analysis, we approximate the $L^\infty$ problem by an $L^p$ one for which the Euler-Lagrange equations make sense, and we compute the Euler-Lagrange equations. In Section \ref{sec:eqns}, we allow $p$ to go to infinity and show that the `limiting' version \eqref{eq:EL1}--\eqref{eq:EL2} of the Euler-Lagrange equations remains when we take the limit, beginning the proof of Theorem \ref{thm:limitingeqns}. In Section \ref{sec:props}, we continue the proof of Theorem \ref{thm:limitingeqns}, showing that the vector field $\varphi$ satisfying equations \eqref{eq:EL1}--\eqref{eq:EL2} is not identically zero and investigating what equations \eqref{eq:EL1}--\eqref{eq:EL2} tell us about $\infty$-elastica. We show that minimisers of $\mathcal{K}_\infty$ exist.


\section{Approximation of the Problem}\label{sec:approx}

\noindent In this section we consider a version of our problem where the $L^\infty$ norm is replaced by an $L^p$ norm, with the view that we will later let $p \rightarrow \infty$ and recover some information about the $L^\infty$ problem. In general, as $p \rightarrow \infty$ we would expect to recover \textit{a} solution of the problem but not necessarily \textit{all} solutions of the problem; to overcome this, we add a penalisation term to our analysis similar to the one in \cite{moser2019structure} (albeit with a slightly different form), so that we can guarantee convergence to any given solution of the problem. We compute the Euler-Lagrange equations for the $L^p$ problem with this penalisation term added. \\

\noindent Recall that we are considering arclength curves $\gamma: [0,L] \rightarrow M$ satisfying the boundary conditions \eqref{eq:bcs} and length constraint \eqref{eq:length}, along with the curvature functional $\mathcal{K}_\infty$. \\

\noindent In practice we will not work directly with Definition \ref{defn:inftyelastica} when dealing with $\infty$-elastica; instead we will make use of another inequality similar to \eqref{eq:inftyelasdefn}. Whenever the inequality in Definition \ref{defn:inftyelastica} is satisfied, so too will this other inequality hold. In this sense, Definition \ref{defn:inftyelastica} is not the strongest possible definition we could use. However, it has the advantage of being easy to understand and the distinction between the two definitions is extremely technical. \\
\noindent Before introducing the inequality we first need to introduce some notation. Fix a compact subset $U$ of $M$ which contains all admissible curves for the problem. For example, the geodesic ball of radius $2L$ centred at the start-point $x_1$ works as such a set. For a curve $\gamma \in \mathcal{G}$ with the $L^\infty$ norm of the curvature given by $\normm{\kappa_\gamma}_{L^\infty} =: K$, we introduce the function
\[
c(K) = \inf_{\hat{\gamma} \in X} \{ \mathcal{L} \left[ \hat{\gamma} \right] \}
\]
where the infimum is taken over the set $X$ of all closed $W^{2,\infty}$ curves $\gamma \subset U$ (closed in the sense that the start-point coincides with the end-point, although the tangents at the start and end need not be the same) satisfying the curvature bound $\normm{\kappa_\gamma}_{L^\infty} \leq K$. Roughly speaking, $c(K)$ is ``the length of the shortest possible loop we can make under the restriction $\normm{\kappa_\gamma}_{L^\infty} \leq K$.'' Note that $K$ is itself a function of $\gamma$ (and so $c$ is too).
For example, in Euclidean space we can take $c(K) = 2 \pi / K$ by Fenchel's theorem. \\ 
For a curve $\gamma \in W^{2,\infty}(0,L)$ with $\norm{\kappa_\gamma}_{L^\infty(0,L)} = K$, we define the \textit{segment of $\gamma$ centred at $s$} by
\[
\text{seg}_{\gamma}(s) = \gamma((s - c(K)/4, s + c(K)/4 )).
\]
That is, $\text{seg}_{\gamma}(s)$ is the image of the restriction of $\gamma$ to the interval $(s - c(K)/4, s + c(K)/4 )$. \\
The \textit{boundary of the segment $\text{seg}_{\gamma}(s)$} is given by the union of the two points
\[
\gamma(s - c(K)/4) \cup \gamma(s + c(K)/4),
\]
where if one of the arguments is outside of the domain $[0,L]$ then the corresponding term is replaced by the empty set. This definition coincides with the conventional definition of the endpoints of a curve, except when the centre of the segment lies close to the endpoints. \\

\noindent One of the key ingredients in this paper is the approximation of the $\mathcal{K}_\infty$ functional by the functionals
\begin{equation*}
\mathcal{K}_p \left[ \gamma \right] = \left( \frac{1}{L} \int_0^L \vert \nabla_T T \vert^p \, \mathrm{d}s \right)^{\frac{1}{p}}
\end{equation*}
where $2 \leq p < \infty$. Although the factor of $1/L$ disappears in the limit, and thus may seem superfluous, it plays a useful role in manipulating some inequalities later on and it makes the $\mathcal{K}_p$ functionals closer to ``averages'' in the traditional sense of the word. \\
Furthermore, given an admissible curve $\tgamma \in \mathcal{G}$ and a real number $\sigma > 0$, we consider the ``penalised $\mathcal{K}_p$ functional''
\begin{align*}
\mathcal{J}_p^\sigma \left[ \gamma ; \tgamma \right] &= \mathcal{K}_p \left[ \gamma \right] + \frac{\sigma}{2L} \int_0^L d \left( \gamma (s), \text{seg}_{\tgamma} (s) \right)^2 \, \mathrm{d}s \\
&= \mathcal{K}_p \left[ \gamma \right] + \mathcal{P}_\sigma \left[ \gamma ; \tgamma \right].
\end{align*}
Using the direct method it can be shown that minimisers of $\mathcal{J}_p^\sigma \left[ \, \bigcdot \, ; \tgamma \right]$ subject to the constraints \eqref{eq:bcs}--\eqref{eq:length} exist. Furthermore, we have
\begin{lem}\label{lem:uniquegeod}
Suppose $\tgamma$ is an $\infty$-elastica and let $(\gamma_p)$ be a sequence of admissible curves minimising $\mathcal{J}_p^\sigma \left[ \,\bigcdot\, ; \tgamma\right]$. \\
1) If $\sigma$ is sufficiently large then the sequence $(\gamma_p)$ converges weakly in $W^{2,q}(0,L)$ to $\tgamma$ as $p \rightarrow \infty$ for every $1 < q < \infty$. \\
2) If $\sigma$ is large enough that part 1) applies, then for $p$ large enough and every $s \in (0,L)$ the minimal geodesic between $\gamma_p(s)$ and $\text{seg}_\tgamma(s)$ is unique, and the end point of this geodesic does not lie on the boundary of the segment $\text{seg}_\tgamma(s)$.
\end{lem}
\begin{proof}
For $q \geq p$, observe the chain of inequalities
\begin{equation}\label{eq:ineqs}
\mathcal{K}_p\left[ \gamma_p \right] \leq \mathcal{J}_p^\sigma \left[ \gamma_p ; \tgamma \right] \leq \mathcal{J}_p^\sigma \left[ \gamma_q ; \tgamma \right] \leq \mathcal{J}_q^\sigma \left[ \gamma_q ; \tgamma \right] \leq \mathcal{J}_q^\sigma \left[ \tgamma ; \tgamma \right] = \mathcal{K}_q \left[ \tgamma \right] \leq \mathcal{K}_\infty \left[ \tgamma \right],
\end{equation}
due to H\"{o}lder's inequality and the choice of the curves $\gamma_p$. It follows that the sequence $(\gamma_p)_{q \leq p < \infty}$ is bounded in $W^{2,q}(0,L)$ for any $2 \leq q < \infty$, so there exists a subsequence $(p_i)$ such that $\gamma_{p_i}$ converges weakly in $W^{2,q}(0,L)$ (and therefore strongly in $C^0$) for every $q < \infty$ to a limit
\[
\gamma_\infty \in \bigcap_{q < \infty} W^{2,q}(0,L).
\]

\noindent The well-known property of the $L^\infty$ norm as the limit of $L^p$ norms then implies that
\begin{align}
\mathcal{J}_\infty^\sigma \left[ \gamma_\infty ; \tgamma \right] &= \lim_{q \rightarrow \infty} \mathcal{J}_q^\sigma \left[ \gamma_\infty ; \tgamma \right] \nonumber \\
&= \lim_{q \rightarrow \infty} \mathcal{K}_q \left[ \gamma_\infty \right] + \mathcal{P}_\sigma \left[ \gamma_\infty ; \tgamma \right]. \nonumber \\
\intertext{By the strong $C^0$ convergence we have that both the $\liminf$ and the limit as $i \rightarrow \infty$ of $\mathcal{P}_\sigma \left[ \gamma_{p_i} ; \tgamma \right]$ coincide and equal $\mathcal{P}_\sigma \left[ \gamma_\infty ; \tgamma \right]$. Therefore, using the lower semicontinuity of the $L^q$ norm with respect to weak convergence, we find that}
\mathcal{J}_\infty^\sigma \left[ \gamma_\infty ; \tgamma \right] &\leq \lim_{q \rightarrow \infty} \liminf_{i \rightarrow \infty} \mathcal{K}_q \left[ \gamma_{p_i} \right] + \lim_{q \rightarrow \infty} \liminf_{i \rightarrow \infty} \mathcal{P}_\sigma \left[ \gamma_\infty ; \tgamma \right] \nonumber \\
&= \lim_{q \rightarrow \infty} \liminf_{i \rightarrow \infty} \mathcal{J}_q^\sigma \left[ \gamma_{p_i} ; \tgamma \right] \label{eq:ineqs2}\\
&\leq \mathcal{K}_\infty \left[ \tgamma \right], \nonumber
\end{align} %
where the final inequality comes from \eqref{eq:ineqs}. By assumption, however, there exists $M \in \mathbb{R}$ such that 
\[
\mathcal{K}_\infty \left[ \tgamma \right] \leq \mathcal{K}_\infty\left[\gamma_\infty\right] + \frac{M}{2L}\int_0^L d(\gamma_\infty(s), \text{seg}_\tgamma(s))^2 \, \mathrm{d}s,
\]
and so 
\[
\mathcal{J}_\infty^\sigma \left[ \gamma_\infty ; \tgamma \right] \leq \mathcal{J}_\infty^M \left[ \gamma_\infty ; \tgamma \right].
\]
Choosing $\sigma > M$ then implies that 
\[
\int_0^L d(\gamma_\infty(s), \text{seg}_\tgamma(s))^2 \, \mathrm{d}s = 0.
\]
We claim that this equality can only be satisfied when $\gamma_\infty = \tgamma$. Indeed, if there is any point on $\gamma_\infty$ which does not lie on $\tgamma$ then by continuity there must be some $\varepsilon > 0$ and some interval in which $d(\gamma_\infty(s), \text{seg}_\tgamma(s))^2 > \varepsilon$, which contradicts the fact that the above integral evaluates to zero. It follows that every point on $\gamma_\infty$ is also a point on $\tgamma$.
\noindent Assume for a contradiction that the two curves are not the same.
\noindent Let $s_0 \in [0,L]$ be the unique number such that $\tgamma(s) = \gamma_\infty(s)$ for all $s \in [0,s_0]$ and such that for all $\varepsilon > 0$ the interval $(s_0, s_0 + \varepsilon)$ contains an element $s$ with $\tgamma(s) \neq \gamma_\infty(s)$. \\
Pick some $\varepsilon > 0$ and define $s_1 = s_0 + \varepsilon$. \\
\noindent When $\varepsilon$ is sufficiently small, there must be infinitely many values $s \in (s_0, s_1)$ such that
\[
\gamma_\infty(s) \not\in \tgamma([ s_0 - c/4, s_1 ]). 
\]
Indeed, for all $s \in (s_0,s_1)$ we have that $\gamma_\infty(s) \not\in \tgamma([ s_0 - c/4, s_0 ]) = \gamma_\infty([ s_0 - c/4, s_0 ])$ because if this inclusion did hold, $\gamma_\infty$ would contain a loop of length less than $c/2$; thus, since the image of $\gamma_\infty$ is a subset of the image of $\tgamma$, $\tgamma$ would also contain such a loop in contradiction of the definition of $c$. \\
Supposing there were only finitely many $s$ such that $\gamma_\infty(s) \not\in \tgamma([ s_0, s_1 ])$, we could enumerate these values as $t_1, \cdots, t_k$. For each $i = 1, \cdots, k$ we could then take a sequence $(u_n)$ approaching $t_i$ such that $\gamma_\infty (u_n) \in \tgamma([ s_0, s_1 ])$, and find that
\[
\gamma_\infty (t_i) = \lim_n \gamma_\infty (u_n) \in \tgamma([ s_0, s_1 ]),
\]
a contradiction.\\
On the other hand, if there were no such $s$-values, it would follow that as sets the inclusion $\gamma_\infty((s_0,s_1)) \subset \tgamma([s_0,s_1])$ holds and hence $\gamma_\infty([s_0,s_1]) \subset \tgamma([s_0,s_1])$ by continuity. For small $\varepsilon$, the curvature bound on $\tgamma$ ensures that $\tgamma([s_0,s_1])$ has no self-intersections, and each curve is parametrised by arclength so we actually have the equality of sets $\gamma_\infty([s_0,s_1]) = \tgamma([s_0,s_1])$. It follows that $\tgamma(s) = \gamma_\infty(s)$ for all $s \in [s_0,s_1]$, contradicting the definition of $s_0$.
\\

\noindent Thus, as we assumed that $\gamma_\infty(s) \in \text{seg}_\tgamma(s) = \tgamma((s-c/4,s+c/4))$ for all $s$, there must be infinitely many values $s \in (s_0,s_1)$ with $\gamma_\infty(s)$ contained in $\tgamma((s_1, s + c/4))$. \\
Now take a sequence $(s_n)$ converging to $s_0$; we find that there is a sequence $(t_n) \subset (s_1, s_1 + c/4]$ such that $\gamma_\infty(s_n) = \tgamma(t_n)$. Taking a subsequence we find that $t_n \rightarrow t_0 \in [s_1, s_1 + c/4]$, with
\[
\tgamma(t_0) = \lim_n \tgamma(t_n) = \lim_n \gamma_\infty(s_n) = \gamma_\infty(s_0) = \tgamma(s_0),
\]
but this implies that $\tgamma$ forms a loop of length less than $c$, which contradicts the definition of $c$.

\noindent For any $1 < q < \infty$, the above argument shows that every subsequence of $(\gamma_p)$ has a further subsequence which converges weakly in $W^{2,q}(0,L)$ to $\tgamma$. Thus the original sequence $(\gamma_p)$ also converges weakly in $W^{2,q}(0,L)$ to $\tgamma$. This proves part 1). \\

\noindent To show part 2), we first note that since $\gamma_p \rightharpoonup \tgamma$ in $W^{2,q}$ we also have the uniform convergence $\gamma_p \rightarrow \tgamma$. \\
By construction, for all $s \in [0,L]$ the segment $\text{seg}_\tgamma(s)$ has no self-intersections. Therefore for each $s \in [0,L]$ there exists $\varepsilon > 0$ such that the open ball $B_{3\varepsilon} (\tgamma(s))$ centred at $\tgamma(s)$ contains only one component of $\text{seg}_\tgamma(s)$.
By the compactness of $[0,L]$ there must then exist an $\varepsilon > 0$ such that for every $s \in [0,L]$, the intersection of the ball $B_{3\varepsilon} (\tgamma(s))$ and the segment $\text{seg}_\tgamma(s)$ consists of only one component.
At the same time, assume that $\varepsilon$ is small enough that the tails of $\text{seg}_\tgamma(s)$ given by
\[
\tgamma \bigr|_{(s - c(\tilde{K})/4, s - c(\tilde{K})/5) \cap (0,L)} \quad \text{ and } \quad \tgamma \bigr|_{(s + c(\tilde{K})/5, s + c(\tilde{K})/4) \cap (0,L)}
\]
lie outside of $B_{3\varepsilon} (\tgamma(s))$ (if $s$ is near either end of $(0,L)$ then one tail may be short or even non-existent). Also assume that $3\varepsilon < r$ where $r$ is the infimum of the injectivity radius over all points $x$ contained in the geodesic ball $B_L(x_0)$.
Now, by the uniform convergence we can take $p$ so large that $d(\gamma_p(s), \tgamma(s)) \leq \varepsilon$ for all $s \in (0,L)$. Then the infimum distance between $\gamma_p(s)$ and $\text{seg}_\tgamma(s)$ is equal to the infimum distance $d_0 \leq \varepsilon$ between the fixed point $\gamma_p(s)$ and the compact set $\tgamma (\lbrack s - c(\tilde{K})/5), s + c(\tilde{K})/5) \rbrack)$, so it must be attained somewhere. Assuming for a contradiction that the minimal distance is attained at more than one point, we obtain two points on the geodesic circle $\partial B = \partial B_{d_0}(\gamma_p(s))$ of radius $d_0 \leq \varepsilon$ at which $\tgamma$ is tangent to $B$, connected by a segment of $\tgamma$, contained entirely inside $B_{3 \varepsilon}(\gamma_p(s))$, of length less than or equal to $2c(\tilde{K})/5$. Picking $\varepsilon$ arbitrarily small, the curvature of $\tgamma$ is then forced to be arbitrarily large, but this is not possible due to the curvature bound $\kappa \leq \tilde{K}$ on $\tgamma$. Hence the minimiser must be unique.
\end{proof}

\noindent When $\tgamma$ is an $\infty$-elastica, Lemma \ref{lem:uniquegeod} ensures that for sufficiently large $p$ it is possible to compute the Euler-Lagrange equations for the minimiser $\gamma_p$ of the functional $\mathcal{J}_p^\sigma \left[ \,\bigcdot\, ; \tgamma \right]$. Indeed, we find that
\begin{prop}\label{prop:EL}
Let $\sigma > 0$ and an $\infty$-elastica $\tgamma \in \mathcal{G}$ be given. Assume $\sigma$ is large enough that part 1) of Lemma \ref{lem:uniquegeod} applies. For every $p$ large enough that part 2) of Lemma \ref{lem:uniquegeod} applies, suppose that $\gamma_p \in W^{2,p}(0,L)$ minimises $\mathcal{J}_{p}^\sigma \left[ \, \bigcdot \, ; \tgamma \right]$ subject to the constraints \eqref{eq:bcs}--\eqref{eq:length}. Let $K_p = \mathcal{K}_p \left[ \gamma_p \right]$ and let $\hvarphi_p = \vert \nabla_{T_p} T_p \vert^{p-2} \nabla_{T_p} T_p$. Then $\hvarphi_p \in W^{2,1}(0,L)$ and there are Lagrange multipliers $\lambda_p \in \mathbb{R}$ such that the Euler-Lagrange equation
\begin{align}\label{eq:ELunnorm}
\begin{split}
0 = &\, \frac{1}{L}K_{p}^{1-p}R(\hvarphi_{p}, T_{p}) T_{p}  + \frac{\sigma}{L} d_{p} \nu_{p} - \lambda_{p} \nabla_{T_{p}} T_{p} + \frac{1}{L}\frac{2p-1}{p}K_{p}^{1-p}\vert \hvarphi_{p} \vert \kappa_{p} \nabla_{T_{p}} T_{p} \\
&+ \frac{1}{L}\frac{2p-1}{p-1}K_{p}^{1-p} \kappa_{p} \vert \hvarphi_{p} \vert' T_{p} - \frac{\sigma}{2L}(d_{p}^2)'T_{p} - \frac{\sigma}{2L}d_{p}^2 \nabla_{T_{p}} T_{p} + \frac{1}{L}K_{p}^{1-p} \nabla_{T_{p}}^2 \hvarphi_{p}
\end{split}
\end{align}
holds a.e. in $(0,L)$, where $d_{p}^2(s) = d ( \gamma_{p} (s), \text{seg}_{\tgamma} (s) )^2$ and $\nu_{p}(s)$ is the unit tangent vector at $\gamma_{p}(s)$ associated with the minimal geodesic from $\text{seg}_{\tgamma} (s)$ to $\gamma_{p}(s)$.
\end{prop}

\noindent To prove this proposition we will need a standard Lemma concerning the derivative of the square of the distance function to a closed set:
\begin{lem}\label{lem:firstvardist}
Let $\Omega \subset M$ be a closed set, and $x_0 \in M$ a point such that there is a unique minimal geodesic connecting $x$ and $\Omega$ (i.e. a unique geodesic from $x$ to $\Omega$ with length equal to $d(x, \Omega)$). Then the function $f(x) = d(x,\Omega)^2$ is differentiable at points $x_0$ where $d(x_0,\Omega)$ is sufficiently small, with the directional derivative in the direction of $v \in T_{x_0}M$ given by $-2d(x,\Omega)\langle v, \nu \rangle$ where $\nu$ is the unit tangent vector in $T_{x_0} M$ to the distance-minimising geodesic from $x_0$ to $\Omega$.
\end{lem}
\noindent Although a statement of precisely the right form is difficult to find in the literature, the result is well-known nevertheless and therefore the proof of this Lemma is omitted.

\begin{proof}[Proof of Proposition \ref{prop:EL}]
Fix $p \in \mathbb{N}$ sufficiently large and let the minimiser of $\mathcal{J}_{p}^\sigma \left[ \, \bigcdot \, ; \tgamma \right]$ be denoted by $\gamma$ (for now, we drop the subscript $p$ from our notation for simplicity). \\
Before computing the first variations it will be useful to establish some notation. Let $t$ be a regular parameter ranging over the interval $[0,1]$ (in particular $t$ need not be the arclength parameter). The speed $v(t)$ of $\gamma$ at time $t$ is given by  $\left\lvert \gamma'(t) \right\rvert$. Take $V$ to be the velocity vector field of $\gamma$, i.e. $V = vT = \frac{\partial \gamma}{\partial t}$.
Consider a smooth variation $\gamma(t,w)$ of $\gamma$, such that $\gamma(t,0) =: \gamma_0(t) = \gamma(t)$, and assume it preserves the boundary conditions of the problem in the sense that for every fixed value of the variational parameter $w$, the curve given as a function of $t$ by $\gamma_w(t) := \gamma(t,w)$ satisfies \eqref{eq:bcs} (but not necessarily \eqref{eq:length}). Note that we are slightly abusing notation here as we use $\gamma$ to refer to both the original curve and its variation, although from context the meaning will be clear. \\
Explicitly, in local co-ordinates we may write $\gamma(t,w) = \gamma_0(t) + w\psi(t)$ for some smooth function $\psi$.
To denote the variational vector field along $\gamma$, write 
\[
W(t) = \frac{\partial \gamma}{\partial w}(t,0).
\]
\noindent Note that since $t$ and $w$ are independent co-ordinates their Lie bracket $[V, W]$ vanishes, i.e. $\nabla_V W = \nabla_W V$. \\
We first compute the first variation of $\mathcal{K}_p^p$. Throughout this calculation we follow the work of \cite{singer2008lectures}. The first variation is given by
\begin{align*}
\frac{\mathrm{d}}{\mathrm{d}w}\Bigr|_{w=0}  \mathcal{K}_p^p\left[ \gamma_w \right] &= \frac{\mathrm{d}}{\mathrm{d}w}\Bigr|_{w=0} \frac{1}{L} \int_0^L \kappa(s)^p \, \mathrm{d}s \\
&= \frac{\mathrm{d}}{\mathrm{d}w}\Bigr|_{w=0} \frac{1}{L} \int_0^1 \kappa(t)^p v(t) \, \mathrm{d}t \\
&= \frac{1}{L} \int_0^1 W(\kappa^p) v + \kappa^p W(v) \, \mathrm{d}t 
\end{align*}
The first term in the integrand becomes
\begin{align*}
W(\kappa^p)v &= W\left(\left\langle \nabla_T T, \nabla_T T \right\rangle^{p/2} \right)v \\
&= p \kappa^{p-2} \left\langle \nabla_W \nabla_T T, \nabla_T T \right\rangle v \\
&= p \kappa^{p-2} \left( \left\langle \nabla_T T, \nabla_T^2 W \right\rangle + \left\langle R(\nabla_T T,T)T , W \right\rangle - 2\kappa^2\left\langle T, \nabla_T W \right\rangle \right)v
\end{align*}
where the third line follows from calculations similar to those in \cite[Equation~9]{singer2008lectures}. Here, $R(\,\bigcdot\, , \,\bigcdot\, ) \,\bigcdot\,$ denotes the Riemann curvature tensor. \\
To compute the second term in the integrand, note that $W(v^2) = 2vW(v)$, and also
\begin{align*}
W(v^2) &= W \left( \left\langle V, V \right\rangle \right) \\
&= 2 \left\langle \nabla_W V, V \right\rangle \\
&= 2 \left\langle \nabla_V W, V \right\rangle \\
&= 2 v^2 \left\langle \nabla_T W, T \right\rangle,
\end{align*}
where the vanishing Lie bracket condition $\nabla_W V = \nabla_V W$ and some linearities of the covariant derivative have been used. Thus $W(v) = \left\langle \nabla_T W, T \right\rangle v$. \\
The first variation of $\mathcal{K}_p^p$ now becomes
\[
\frac{1}{L}\int_0^1 \left(p \kappa^{p-2} \left( \left\langle \nabla_T T, \nabla_T^2 W \right\rangle + \left\langle R(\nabla_T T,T)T , W \right\rangle - 2\kappa^2\left\langle T, \nabla_T W \right\rangle \right) + \kappa^p \left\langle \nabla_T W, T \right\rangle\right) v \, \mathrm{d}t,
\]
and reparametrising by arclength turns this into
\begin{equation*}
\frac{1}{L}\int_0^L p \kappa^{p-2} \left( \left\langle \nabla_T T, \nabla_T^2 W \right\rangle + \left\langle R(\nabla_T T,T)T , W \right\rangle \right) + (1-2p) \kappa^p \left\langle T, \nabla_T W \right\rangle \, \mathrm{d}s. \\
\end{equation*}
We now compute the first variation of the penalisation term $\mathcal{P}_\sigma\left[\,\bigcdot\, ; \tgamma\right]$. This is given by
\begin{align*}
\frac{\mathrm{d}}{\mathrm{d}w}\Bigr|_{w=0} \mathcal{P}_\sigma\left[\gamma_w ; \tgamma\right] &= \frac{\mathrm{d}}{\mathrm{d}w}\Bigr|_{w=0} \frac{\sigma}{2L} \int_0^L d(s)^2 \, \mathrm{d}s \\
&= \frac{\mathrm{d}}{\mathrm{d}w}\Bigr|_{w=0} \frac{\sigma}{2L} \int_0^1 d(t)^2 v(t) \, \mathrm{d}t \\
&= \frac{\sigma}{2L} \int_0^1 W(d^2) v + d^2 W(v) \, \mathrm{d}t,
\end{align*}
where the shorthand $d(t)$ has been used to denote the distance term for brevity.
It remains to obtain an expression for $W(d^2)$. The instantaneous direction of movement of the point $\gamma(t)$ is given by $\partial \gamma / \partial w \bigr|_{w=0} \gamma(t,w) = W(t)$. Now, by taking $w$ sufficiently small the difference between the arclength along $\gamma_w$ up to the point $\gamma_w(t)$ and the arclength along $\gamma$ up to the point $\gamma(t)$ can be made arbitrarily small, such that the endpoints of $\text{seg}_\tgamma (\tilde{s} \bigr|_{\tilde{s} = s(t,w) L / L_{\gamma_w}} )$ move so little as to leave unchanged the portion of $\text{seg}_\tgamma (\tilde{s} \bigr|_{\tilde{s} = s(t,w) L / L_{\gamma_w}} )$ contained within $B_{3\varepsilon} (\tgamma(s))$. In effect, for small $w$ and fixed $t$, the distance term is equal to the distance between $\gamma(t,w)$ and a fixed set of points. Using Lemma \ref{lem:firstvardist} we therefore obtain that $W(d^2) = -2d\langle \nu, W \rangle$, where $\nu$ is the tangent vector at $\gamma(t)$ associated with the distance-minimising curve starting on $\text{seg}_\tgamma (\tilde{s} \bigr|_{\tilde{s} = s(t) L / L_{\gamma}} )$ and ending at $\gamma(t)$. \\
The $W(v)$ term was already computed when considering the first variation of $\mathcal{K}_p^p$, and so we can substitute in the expressions for $W(d^2)$ and $W(v)$ to turn the first variation of $\mathcal{P}_\sigma\left[\,\bigcdot\, ; \tgamma\right]$ into
\[
\frac{\sigma}{2L} \int_0^1 \left(  2d \left\langle \nu, W \right\rangle + d^2 \left\langle \nabla_T W, T \right\rangle \right)v \, \mathrm{d}t.
\]
Again, the presence of $v$ makes it possible to reparametrise this by arclength, yielding
\begin{equation*}
\frac{\sigma}{2L} \int_0^L \left(  2d \left\langle \nu, W \right\rangle + d^2 \left\langle \nabla_T W, T \right\rangle \right) \, \mathrm{d}s. \\
\end{equation*}
Finally, to account for the length constraint \eqref{eq:length} we need to apply the Lagrange Multiplier Principle. We compute the first variation of the length functional
\[
\mathcal{L}\left[ \gamma \right] = \int_0^1 \vert \gamma'(t) \vert \, \mathrm{d}t = \int_0^1 v(t) \, \mathrm{d}t.
\]
as
\begin{align}
\frac{\mathrm{d}}{\mathrm{d}w}\Bigr|_{w=0} \mathcal{L}\left[ \gamma_w \right] &= \frac{\mathrm{d}}{\mathrm{d}w}\Bigr|_{w=0} \int_0^1 v(t) \, \mathrm{d}t \nonumber \\
&=  \int_0^1 W(v) \, \mathrm{d}t \nonumber \\
&= \int_0^1 \left\langle \nabla_T W, T \right\rangle  v  \, \mathrm{d}t \nonumber \\
&= \int_0^L \left\langle \nabla_T W, T \right\rangle \, \mathrm{d}s \nonumber
\end{align}
so the Euler Lagrange equation will have a term of the form
\[
\lambda \int_0^L \left\langle \nabla_T W, T \right\rangle \, \mathrm{d}s
\]
added to it, for some constant $\lambda \in \mathbb{R}$ (the Lagrange multiplier). \\
Combining the three calculations of the first variations and reintroducing the subscript $p$ to our notation, we obtain the following Euler-Lagrange equation in the weak sense:
\begin{equation*}\label{Lagrangep}
\int_0^L \left\langle E_1 , W \right \rangle + \left\langle E_2 , \nabla_{T_p} W \right \rangle + \left\langle E_3 , \nabla_{T_p}^2 W \right \rangle \, \mathrm{d}s = 0,
\end{equation*}
satisfied for every test variation $W$ (smooth and compactly supported), where
\begin{align*}
E_1 &=  \frac{1}{L}K_p^{1-p}R(\hvarphi_p, T_p) T_p  + \frac{\sigma}{L} d_p \nu_p ,\\
E_2 &=  \frac{1}{L}\frac{2p-1}{p} K_p^{1-p} \vert \hvarphi_p \vert^{p/(p-1)} T_p + \lambda_p T_p + \frac{\sigma}{2L}d_p^2 T_p ,
\intertext{and}
E_3 &=  \frac{1}{L}K^{1-p}\hvarphi_p.
\end{align*}
After writing out the above Euler-Lagrange equation in local co-ordinates, we can make use of standard regularity results to deduce that $\hvarphi_p \in W^{2,1}$, and this regularity allows us to integrate the Euler-Lagrange equation by parts:
\begin{equation*}
\int_0^L \left\langle E_1 , W \right \rangle - \left\langle \nabla_{T_p} E_2 , W \right \rangle + \langle \nabla_{T_p}^2 E_3 , W \rangle \, \mathrm{d}s = 0,
\end{equation*}
having used the fact that $d_p^2$ is Lipschitz continuous with Lipschitz constant 1 and thus is in $W^{1,\infty}$. The Fundamental Theorem of Calculus of Variations then implies that the Euler-Lagrange equation \eqref{eq:ELunnorm} holds a.e. in $(0,L)$, finishing the proof.
\end{proof}

\noindent Proposition \ref{prop:EL} highlights a fundamental difference between $p$-elastica and $\infty$-elastica: when $p$ is finite the fact that $\hvarphi_p$ lies in $W^{2,1}(0,L)$ means that the curvature vector is continuous, yet the curvature vector of an $\infty$-elastica need not be continuous and indeed in some cases is guaranteed by the constraints of the problem and the results of Theorem \ref{thm:limitingeqns} to be discontinuous.


\section{Limiting Equations for $\infty$-elastica}\label{sec:eqns}

\noindent In this section we show that there is a limiting curve which the $p$-minimisers converge to, along which a limiting version of a normalisation of the Euler-Lagrange equation \eqref{eq:ELunnorm} is satisfied. In addition we derive another equation from our definition of the rescaled curvature vector; these two equations together act as a system of equations that can be considered in a sense to be Euler-Lagrange equations for the $L^\infty$ problem. This proves most of statement 1) of Theorem \ref{thm:limitingeqns}, except that the limiting $\varphi$ is not identically zero. \\

\noindent Before we can obtain the limiting Euler-Lagrange equations \eqref{eq:EL1}--\eqref{eq:EL2}, we must know that the terms in \eqref{eq:ELunnorm} do not become too large when $p \rightarrow \infty$. To this end, we prove two lemmata. The first lemma concerns the size of the $\mathcal{K}_p$ functionals when evaluated at minimisers of $\mathcal{J}_p^\sigma \left[ \,\bigcdot\, ; \tgamma\right]$.
\begin{lem}\label{lem:Kpconv}
Suppose $\tgamma$ is an $\infty$-elastica and let $(\gamma_p)$ be a sequence of admissible curves minimising $\mathcal{J}_p^\sigma \left[ \,\bigcdot\, ; \tgamma\right]$. Assume that $\sigma$ is large enough that both parts of Lemma \ref{lem:uniquegeod} apply. Then the sequence $(K_p)$ converges to $K := \mathcal{K}_\infty \left[ \tgamma \right]$.
\end{lem}
\begin{proof}
Take $\sigma$ large enough so the situation in the proof of Lemma \ref{lem:uniquegeod} applies. Consider the inequalities in \eqref{eq:ineqs}. They show that the sequence of real numbers given by $(\mathcal{J}_p^\sigma\left[ \gamma_p ; \tgamma \right])$ is increasing and bounded above by $\mathcal{K}_\infty \left[ \tgamma \right]$, hence convergent with
\[
\lim_{p \rightarrow \infty} \mathcal{J}_p^\sigma\left[ \gamma_p ; \tgamma \right] \leq \mathcal{K}_\infty \left[ \tgamma \right].
\]
However, the inequalities in \eqref{eq:ineqs2} combined with the fact that $\tgamma = \gamma_\infty$ mean that
\[
\mathcal{K}_\infty \left[ \tgamma \right] = \mathcal{J}_\infty^\sigma \left[ \tgamma ; \tgamma \right] \leq \lim_{q \rightarrow \infty} \liminf_{i \rightarrow \infty} \mathcal{J}_q^\sigma \left[ \gamma_{p_i} ; \tgamma \right] \leq \liminf_{i \rightarrow \infty} \mathcal{J}_{p_i}^\sigma \left[ \gamma_{p_i} ; \tgamma \right] = \lim_{p \rightarrow \infty} \mathcal{J}_p^\sigma\left[ \gamma_p ; \tgamma \right].
\]
It follows that $\lim_{p \rightarrow \infty} \mathcal{J}_p^\sigma\left[ \gamma_p ; \tgamma \right] = \mathcal{K}_\infty \left[ \tgamma \right]$. Now, 
\[
\mathcal{K}_\infty \left[ \tgamma \right] = \lim_{p \rightarrow \infty} \mathcal{J}_p^\sigma\left[ \gamma_p ; \tgamma \right] = \lim_{p \rightarrow \infty} \mathcal{K}_p \left[ \gamma_p \right] + \lim_{p \rightarrow \infty} \frac{\sigma}{2L} \int_0^L d(\gamma_p(s),\text{seg}_\tgamma (s))^2 \, \mathrm{d}s,
\]
and by the uniform convergence of $\gamma_p$ to $\tgamma$ the second term on the right hand side vanishes, finishing the proof.
\end{proof}

\noindent The second lemma concerns the size of the Lagrange multipliers $\lambda_p$. 

\begin{lem}\label{lem:lagrangebound}
Suppose $\tgamma$ is an $\infty$-elastica, and that $\sigma$ is as large as required by Lemma \ref{lem:uniquegeod}. Then the sequence of Lagrange multipliers $(\lambda_p)$ in the Euler-Lagrange equation \eqref{eq:ELunnorm} is bounded.
\end{lem}

\noindent Before proving this lemma, we introduce some common but not ubiquitous notation: given sequences $(\alpha_p)$, $(\beta_p)$ we say that $\alpha_p \lesssim \beta_p$ if there exists a constant $c$ independent of $p$ such that $\alpha_p \leq c \beta_p$ for every $p \in \mathbb{N}$.

\begin{proof}[Proof of Lemma \ref{lem:lagrangebound}]
Rearranging \eqref{eq:ELunnorm} gives us the equation
\begin{equation}\label{eq:lagbound1}
\tilde{E}_p  = L \lambda_p \nabla_{T_p} T_p,
\end{equation}
where for $K_p = \mathcal{K}_p \left[ \gamma_p \right]$ we have
\begin{align*}
\tilde{E}_p := &K_p^{1-p}\nabla_{T_p}^2 \left( \kappa_p^{p-2} \nabla_{T_p} T_p \right) + K_p^{1-p}\kappa_p^{p-2} R(\nabla_{T_p} T_p, T_p ) T_p + \sigma d_p \nu_p \\
&+ K_p^{1-p}\frac{2p-1}{p} \nabla_{T_p} \left( \kappa_p^p T_p \right) - \frac{\sigma}{2} (d_p^2)' T - \frac{\sigma}{2} d_p^2 \nabla_{T_p} T_p.
\end{align*}
Note here that we have combined two of the terms from \eqref{eq:ELunnorm} into a single term expressed as a derivative:
\[
\frac{1}{L}\frac{2p-1}{p}K_{p}^{1-p}\vert \hvarphi_{p} \vert \kappa_{p} \nabla_{T_{p}} T_{p} + \frac{1}{L}\frac{2p-1}{p-1}K_{p}^{1-p} \kappa_{p} \vert \hvarphi_{p} \vert' T_{p} =  \frac{1}{L} \frac{2p-1}{p} \nabla_{T_{p}} \left( \langle \hvarphi_p, \nabla_{T_p} T_p \rangle T_p \right).
\]
If eventually the $\lambda_p$s all become 0, the lemma is immediate, so assume this is not the case and discard from the sequence any $\lambda_p$s which are zero. Then
\[
\frac{1}{L \lambda_p} \tilde{E}_p = \nabla_{T_p} T_p \rightharpoonup \nabla_T T \neq 0 \text{ in }L^q(0,L)\text{ for any }q < \infty,
\]
recalling that $\nabla_T T$ is not identically zero because it is assumed a priori that no geodesic solutions exist. Since $\nabla_T T$ does not vanish, there exists a test vector field $\psi$ such that $\int_0^L \left\langle \nabla_T T, \psi \right\rangle \, \mathrm{d}s = 1$. Without loss of generality, we can assume that the support of $\psi$ is contained in a single co-ordinate chart and thus we can write $\psi(s) = \psi^i \vecpar{i} \rvert_{\gamma(s)}$. In these co-ordinates, for sufficiently large $p \in \mathbb{N}$, consider the test vector field $\psi_p$ defined along $\gamma_p$ by $\psi^i \vecpar{i} \rvert_{\gamma_p(s)}$ (this is valid due to the uniform convergence of $\gamma_p$ to $\gamma$), so the coefficients $\psi^i$ are smooth and do not depend on $p$.
Taking the inner product of $\tilde{E}_p$ with $\psi_p$, integrating from 0 to $L$ and integrating by parts tells us that
\begin{align*}
\int_0^L \left\langle \tilde{E}_p, \psi_p \right\rangle \, \mathrm{d}s &\leq I_1 + I_2 + I_3 + I_4 + I_5 + I_6,
\end{align*}
where
\begin{align*}
I_1 &= K_p^{1-p}\int_0^L \kappa_p^{p-2} \vert \nabla_{T_p} T_p \vert \vert \nabla_{T_p}^2 \psi_p  \vert  \, \mathrm{d}s, \\
I_2 &= K_p^{1-p}\int_0^L \kappa_p^{p-2} \vert R\left(\nabla_{T_p} T_p, T_p \right) T_p \vert \vert \psi_p  \vert \, \mathrm{d}s, \\
I_3 &= \sigma \int_0^L d_p \vert \nu_p  \vert \vert \psi_p  \vert  \, \mathrm{d}s, \\
I_4 &= 2K_p^{1-p}\int_0^L \vert \kappa_p^{p-2} \nabla_{T_p} T_p \vert \vert \nabla_{T_p} \psi_p  \vert \, \mathrm{d}s, \\
I_5 &= \frac{\sigma}{2} \int_0^L \vert (d_p^2)' \vert  \vert \psi_p  \vert  \, \mathrm{d}s, \\
\intertext{and}
I_6 &= \frac{\sigma}{2} \int_0^L d_p^2 \vert \nabla_{T_p} T_p \vert \vert \psi_p \vert \, \mathrm{d}s.
\end{align*}
Using the definition of $\psi_p$, H\"{o}lder's inequality, the definitions of $\kappa_p$ and $K_p$, Lemma \ref{lem:Kpconv}, the boundedness of $R$ on compact subsets of $M$, and the boundedness of $d_p^2$ and $(d_p^2)'$, it follows that each of the six integrals $I_n$ is bounded above by some constant $C(\psi)$ depending on the components $\psi^i$ of $\psi$ but independent of $p$, and hence so too is $\int_0^L \langle \tilde{E}_p, \psi \rangle \, \mathrm{d}s$. \\
For example, we have that
\begin{align*}
I_2 &\lesssim \norm{\psi_p}_{L^\infty(0,L)} K_p^{1-p} \int_0^L \kappa_p^{p-1} \, \mathrm{d}s \\
&\lesssim C(\psi) K_p^{1-p} \norm{\kappa_p}_{L^p}^{p-1} \norm{1}_{L^p} \\
&\lesssim C(\psi),
\end{align*}
where we have made the dependence on $\psi$ explicit. Here we have used the boundedness of $R$ and $\psi$ and the definition of $\kappa_p$ in the first line, H\"{o}lder's inequality in the second line, and the definition of $K_p$ in the third line.
The other five integrals are treated similarly. \\
By the weak $L^q(0,L)$ convergence of $\nabla_{T_p} T_p$ and the choice of $\psi$,
\[
\int_0^L \left\langle \frac{1}{L \lambda_p} \tilde{E}_p, \psi_p \right\rangle \, \mathrm{d}s \geq \frac{1}{2}
\]
for sufficiently large $p$, i.e.
\[
L \lambda_p \leq 2\int_0^L \left\langle \tilde{E}_p, \psi_p \right\rangle \, \mathrm{d}s \leq 2 C(\psi).
\]
This gives the desired bound.
\end{proof}
\noindent With these two lemmata established we are well-positioned to begin to prove Theorem \ref{thm:limitingeqns}. \\

\noindent \textit{Beginning of proof of Theorem \ref{thm:limitingeqns}.}
We begin by normalising equation \eqref{eq:ELunnorm}, setting $\varphi_p = K_p^{1-p} \hvarphi_p$. Thus we obtain the equation
\begin{align}\label{eq:ELnorm}
\begin{split}
0 = &\, R(\varphi_p, T_p) T_p  + \sigma d_p \nu_p - L \lambda_p \nabla_{T_p} T_p + \frac{2p-1}{p}\vert \varphi_p \vert \kappa_p \nabla_{T_p} T_p \\
&+ \frac{2p-1}{p-1} \kappa_p \vert \varphi_p \vert' T_p - \frac{\sigma}{2} (d_p^2)' T_p - \frac{\sigma d_p^2}{2} \nabla_{T_p} T_p +  \nabla_{T_p}^2 \varphi_p.
\end{split}
\end{align}
Our goal now is to obtain higher regularity of the $\varphi_p$ terms along with bounds on them which are uniform in $p$. \\
Choose $\psi \in C^\infty_c(0,L)$ to be a test function such that
\[
\norm{ \frac{\psi'(s)^2}{\psi(s)} }_{L^\infty(0,L)} < \infty.
\]
For example, the square of any test function suffices here, as does a rescaling of the standard ``bump'' function given by $\exp(-1/(1-x^2))$. 
Now consider the function $f(s) = \frac{1}{2}\vert\varphi_p(s)\vert^2$. Observing that $\int_0^L (f' \psi^2)' \, \mathrm{d}s = 0$ then integrating by parts shows that
\begin{align*}
0 &= \int_0^L \psi^2 f'' - 2 \psi \psi'' f - 2 (\psi')^2 f \, \mathrm{d}s.
\end{align*}
Substituting in the definition of $f$ gives
\begin{equation*}
0 = \int_0^L \psi^2\hiip{\varphi_p}{\nabla_{T_p}^2 \varphi_p} + \vert \psi \nabla_{T_p} \varphi_p \vert^2 -\psi \psi'' \vert \varphi_p \vert^2 - (\psi')^2 \vert \varphi_p \vert^2   \, \mathrm{d}s.
\end{equation*}
Again substituting, this time using the Euler-Lagrange equation \eqref{eq:ELnorm} to replace the $\nabla_{T_p}^2 \varphi_p$ term, we find that
\begin{equation}\label{eq:Is}
\int_0^L \vert \psi \nabla_{T_p} \varphi_p \vert^2 \, \mathrm{d}s \leq I_1 + I_2 + I_3 + I_4 + I_5 + I_6 + I_7,
\end{equation}
where
\begin{align*}
I_1 &= \int_0^L \psi^2 \vert R(\varphi_p,T_p)T_p \vert\vert \varphi_p \vert  \, \mathrm{d}s ,\\
I_2 &= \sigma\int_0^L \psi^2 d_p \vert \varphi_p \vert \, \mathrm{d}s ,\\
I_3 &= L \lambda_p \int_0^L  \psi^2 \vert \varphi_p \vert\vert \nabla_{T_p} T_p \vert \, \mathrm{d}s ,\\
I_4 &= \frac{2p - 1}{p} \int_0^L \psi^2 \kappa_p \vert \varphi_p \vert^2 \vert \nabla_{T_p} T_p \vert  \, \mathrm{d}s ,\\
I_5 &= \frac{\sigma}{2} \int_0^L  \psi^2 d_p^2 \vert \varphi_p \vert \vert \nabla_{T_p} T_p \vert \, \mathrm{d}s, \\
I_6 &= \int_0^L \vert \psi \vert \vert \psi'' \vert \vert \varphi_p \vert^2  \, \mathrm{d}s, \\
\intertext{and}
I_7 &= \int_0^L \vert \psi' \vert^2 \vert \varphi_p \vert^2  \, \mathrm{d}s. 
\end{align*}
The contributions from all the terms in \eqref{eq:ELnorm} parallel to $T_p$ vanish since we take the inner product $\hiip{\varphi_p}{\nabla_{T_p}^2 \varphi_p}$ and $T_p$ is orthogonal to $\varphi_p$. \\
\noindent The first integral can be estimated by
\begin{align*}
I_1 &\lesssim \int_0^L  \psi ^2 \vert \varphi_p \vert^2 \, \mathrm{d}s \\
&\lesssim K_p^{2-2p} \int_0^L \psi^2 \kappa_p^{2p-2} \, \mathrm{d}s \\ 
&\lesssim K_p^{2-2p} \norm{\psi}_{L^\infty} \norm{\psi \kappa_p^{p-2}}_{L^\infty}  \int_0^L \kappa_p^p \, \mathrm{d}s \\
&\lesssim K_p^{2-p} \norm{\psi}_{L^\infty} \norm{\psi \kappa_p^{p-2}}_{L^\infty}  \\
&\lesssim  \norm{\psi}_{L^\infty}^{\frac{p}{p-1}} \norm{\psi \varphi_p}_{L^\infty}^{\frac{p-2}{p-1}} \\ 
\intertext{having used the definitions of $\varphi_p$ and $K_p$ as well as the boundedness of the Riemann curvature tensor. Now, since $\norm{f}_{L^\infty} \leq \norm{f}_{L^1}/L + \norm{f'}_{L^1}$ (this can be seen e.g. by the Sobolev embedding of $W^{1,\infty}(0,L)$ into $L^\infty(0,L)$) we get that}
I_1 &\lesssim  \norm{\psi}_{L^\infty}^{\frac{p}{p-1}} \left( \sqrt{L} \norm{\psi \nabla_{T_p} \varphi_p}_{L^2} + \norm{\varphi_p}_{L^1} \normm{\psi'}_{L^\infty} + \frac{1}{L} \norm{\psi \varphi_p}_{L^1} \right)^{\frac{p-2}{p-1}} \\
&\lesssim  \norm{\psi}_{L^\infty}^{\frac{p}{p-1}} \left(\norm{\psi \nabla_{T_p} \varphi_p}_{L^2} + \norm{\varphi_p}_{L^1} \normm{\psi'}_{L^\infty} + \norm{\psi \varphi_p}_{L^1} \right)^{\frac{p-2}{p-1}} \\
&\lesssim  \norm{\psi}_{L^\infty}^{\frac{p}{p-1}}\left(\norm{\psi \nabla_{T_p} \varphi_p}_{L^2}^{\frac{p-2}{p-1}} +  \left( \norm{\varphi_p}_{L^1} \normm{\psi'}_{L^\infty} + \norm{\varphi_p}_{L^1}\norm{\psi}_{L^\infty}  \right)^{\frac{p-2}{p-1}}\right)  \\
&\lesssim   \norm{\psi}_{L^\infty}^{\frac{p}{p-1}}\norm{\psi \nabla_{T_p} \varphi_p}_{L^2}^{\frac{p-2}{p-1}} + \normm{\psi'}_{L^\infty}^{\frac{p-2}{p-1}} \norm{\psi}_{L^\infty}^{\frac{p}{p-1}} + \norm{\psi}_{L^\infty}^2 \\
\intertext{because of the uniform $L^1$ bound on $\varphi_p$ which comes from its definition and an immediate application of H\"{o}lder's inequality. Applying Young's inequality to the first term of this expression, we find that}
I_1 &\leq \varepsilon \frac{p-2}{2p-2} \norm{\psi \nabla_{T_p} \varphi_p}_{L^2}^2 + C \left( \varepsilon^{\frac{2-p}{p}} \frac{p}{2p-2} \norm{\psi}_{L^\infty}^{2} + \normm{\psi'}_{L^\infty}^{\frac{p-2}{p-1}} \norm{\psi}_{L^\infty}^{\frac{p}{p-1}} + \norm{\psi}_{L^\infty}^2 \right) \\
&\leq \varepsilon \norm{\psi \nabla_{T_p} \varphi_p}_{L^2}^2 + C \left( \varepsilon^{\frac{2-p}{p}}\norm{\psi}_{L^\infty}^{2}  + \normm{\psi'}_{L^\infty}^{\frac{p-2}{p-1}} \norm{\psi}_{L^\infty}^{\frac{p}{p-1}} + \norm{\psi}_{L^\infty}^2 \right) \\
\intertext{where $C \in \mathbb{R}$ is some constant independent of $p$. We have therefore shown that}
I_1 &\leq \varepsilon \norm{\psi \nabla_{T_p} \varphi_p}_{L^2}^2 + C,
\end{align*}
where $C$ is a constant which depends on $\varepsilon$ and $\psi$ but importantly not on $p$. \\
Similarly, the fourth integral can be estimated by
\begin{align*}
I_4 &\leq C \norm{\psi^{\frac{p-2}{p-1}}}_{L^\infty} \normm{\kappa_p^p \psi^{\frac{p}{p-1}}}_{L^\infty} K_p^{2-2p} \int_0^L \kappa_p^p \, \mathrm{d}s \\
&\lesssim \norm{\psi}^{\frac{p-2}{p-1}}_{L^\infty} K_p^{2-p} \norm{\kappa_p^{p-1} \psi}_{L^\infty}^{\frac{p}{p-1}} \\
&\lesssim \norm{\psi}^{\frac{p-2}{p-1}}_{L^\infty} \norm{\psi \varphi_p }_{L^\infty}^{\frac{p}{p-1}} \\
&\lesssim \norm{\psi}_{L^\infty}^{\frac{p-2}{p-1}} \left( \norm{\psi \nabla_{T_p} \varphi_p}_{L^2} + \normm{\psi' \varphi_p}_{L^1} + \norm{\psi \varphi_p}_{L^1} \right)^{\frac{p}{p-1}} \\ 
&\lesssim \norm{\psi}_{L^\infty}^{\frac{p-2}{p-1}}\norm{\psi \nabla_{T_p} \varphi_p}_{L^2}^{\frac{p}{p-1}} + \normm{\psi'}_{L^\infty}^{\frac{p}{p-1}}\norm{\psi}_{L^\infty}^{\frac{p-2}{p-1}} + \norm{\psi}_{L^\infty}^2 \\
&\lesssim \varepsilon \norm{\psi \nabla_{T_p} \varphi_p}_{L^2}^{2} + \varepsilon^{\frac{-p}{p-2}}\norm{\psi}_{L^\infty}^{\frac{-2}{p-1}} + \normm{\psi'}_{L^\infty}^{\frac{p}{p-1}}\norm{\psi}_{L^\infty}^{\frac{p-2}{p-1}} + \norm{\psi}_{L^\infty}^2 
\intertext{which again results in an inequality of the form}
I_4 &\leq \varepsilon \norm{\psi \nabla_{T_p} \varphi_p}_{L^2}^2 + C.
\end{align*}
The sixth and seventh integrals yield similar estimates, where in the calculations for $I_7$ we make use of the assumption on the form of $\psi$:
\begin{align*}
I_7 &= \int_0^L \vert \psi' \vert^2 \vert \varphi_p \vert^2  \, \mathrm{d}s \\
&= \int_0^L \frac{\vert \psi' \vert^2}{\vert \psi \vert} \vert \varphi_p \vert \vert \psi \varphi_p \vert \, \mathrm{d}s \\
&\leq \norm{\frac{(\psi')^2}{\psi}}_{L^\infty} \norm{\psi \varphi_p}_{L^\infty} \norm{\varphi_p}_{L^1},
\end{align*}
and from this point the calculations with the $\norm{\psi \varphi_p}_{L^\infty}$ follow similarly to the calculations for $I_1$. \\
\noindent We also compute from their definitions that
\[
I_2, I_3, I_5 \leq C(\psi)
\]
using Lemma \ref{lem:Kpconv} and the uniform $L^1$ bound on $\kappa_p \varphi_p$ coming from its definition. \\
\noindent Combining these seven estimates and picking $\varepsilon$ sufficiently small, we obtain a uniform bound for $\normm{\psi \nabla_{T_p} \varphi}_{L^2}^2$. By picking $\psi$ such that $\psi \geq 1$ on $[a,b] \subset (0,L)$, we get a bound on $\normm{\nabla_{T_p} \varphi_p}_{L^2([a,b])}^2$ which is uniform in $p$. From the Sobolev Embedding Theorem we obtain a uniform bound on $\normm{\varphi_p}_{L^\infty([a,b])}$, and then going back to the Euler-Lagrange equation \eqref{eq:ELnorm} we obtain a uniform $L^2([a,b])$ bound on $\nabla_{T_p}^2 \varphi_p$ by expressing it as the sum of terms each possessing such a uniform bound. Applying the Sobolev Embedding Theorem again we get a uniform bound on $\normm{\nabla_{T_p} \varphi_p}_{L^\infty([a,b])}$, and again going back to \eqref{eq:ELnorm} we get a uniform bound on $\normm{\nabla_{T_p}^2 \varphi_p}_{L^\infty([a,b])}$. \\
From this we obtain a subsequence of $(\varphi_p)$ which converges weakly in $W^{2,q}([a,b])$ to a limit $\varphi \in W^{2,\infty}([a,b])$ for every $q < \infty$. Since $[a,b]$ was arbitrary we have that in fact $\varphi \in W^{2,\infty}_{\text{loc}}(0,L)$. Now, the sequence $(\lambda_p)$ is bounded by Lemma \ref{lem:lagrangebound} and so has a convergent subsequence, and by the uniform convergence of the curves $\gamma_p$ to $\gamma$ and the distance functions $d_p$ to $0$ we see after expressing \eqref{eq:ELnorm} in local co-ordinates and taking a suitable subsequence (not explicitly labelled) that we obtain the limiting equation \eqref{eq:EL1}, with the final term coming from the fact that
\[
\frac{2p-1}{p} \nabla_{T_p} \left( \langle \varphi_p, \nabla_{T_p} T_p \rangle T_p \right) = \frac{2p-1}{p}K_{p}^{1-p}\vert \hvarphi_{p} \vert \kappa_{p} \nabla_{T_{p}} T_{p} + \frac{2p-1}{p-1}K_{p}^{1-p} \kappa_{p} \vert \hvarphi_{p} \vert' T_{p}.
\]

\noindent To obtain equation \eqref{eq:EL2}, we consider
\[
\nabla_{T_p} T_p = \vert \varphi_p \vert^{1/(p-1)} K_p \frac{\varphi_p}{\vert \varphi_p \vert}
\]
and let $[a,b] \subset (0,L)$ be arbitrary. By Lemma \ref{lem:Kpconv}, $K_p \rightarrow K$. As $\varphi$ is continuous, and the convergence of the $\varphi_p$s to $\varphi$ is locally uniform, for any fixed $t$ with $\varphi(t) \neq 0$ there exist $\delta, \varepsilon > 0$ such that $\delta \leq \vert \varphi_p \vert \leq 1/\delta$ in $(t - \varepsilon, t + \varepsilon) \cap [a,b]$ for $p$ sufficiently large. Thus
\[
\frac{\varphi_p}{\vert \varphi_p \vert} \rightarrow \frac{\varphi}{\vert \varphi \vert} \text{ uniformly in }(t - \varepsilon, t + \varepsilon) \cap (a,b).
\]
Also, since $\delta \leq \vert \varphi_p \vert\, \leq 1/\delta$ in $(t - \varepsilon, t + \varepsilon) \cap (a,b)$, it holds in this interval that $\vert \varphi_p \vert^{1/(p-1)} \rightarrow 1$ uniformly as $p \rightarrow \infty$. \\
Combining these convergence results, \eqref{eq:EL2} is obtained at the points in $(a,b)$ where $\vert \varphi \vert\, \neq 0$. Because $[a,b] \subset (0,L)$ was arbitrary and \eqref{eq:EL2} holds wherever $\varphi = 0$, it must therefore hold a.e. in $(0,L)$.

\section{Properties of $\infty$-elastica and Existence of $\mathcal{K}_\infty$ Minimisers}\label{sec:props}

\noindent In this section we examine the limiting system of equations \eqref{eq:EL1}--\eqref{eq:EL2} and how they relate to solutions of our problem. We prove the remainder of the first statement of Theorem \ref{thm:limitingeqns}-- that the limiting vector field $\varphi$ is not identically zero-- and we finish the proof of Theorem \ref{thm:limitingeqns} by proving the second statement: that if a curve admits a solution of \eqref{eq:EL1}--\eqref{eq:EL2}, then its curvature may take on at most two values. We show the existence of minimisers of $\mathcal{K}_\infty$.

\begin{proof}[Proof of Theorem \ref{thm:limitingeqns} continued]

\noindent To finish proving statement 1) of Theorem \ref{thm:limitingeqns}, it remains only to show that $\varphi$ is not identically zero. To this end, let the number $0 < \delta < L/2$ be given and take a function $\eta \in C^\infty([0,L])$ such that $\eta(0) = \eta(L) = 1$, $\vert \eta \vert \,\leq 1$, and $\text{supp}(\eta) \subseteq [0,\delta] \cup [L - \delta, L]$. Choose a sequence of vector fields $(Z_p)$ with each vector field $Z_p$ defined along the curve $\gamma_p$ such that $Z_p(0) = Z_p(L) = 0$ with $\nabla_{T_p} Z_p = T_p$ on $[0,\delta] \cup [L - \delta, L]$ (such a choice is made possibly by the standard existence theorem for ODEs with initial data). Moreover, take the vector fields $Z_p$ to have continuous covariant derivatives up to second order on $(0,L)$. \\
By the definitions of $\varphi_p$ and $Z_p$,
\begin{align*}
K_p^{1-p} \int_0^L \eta \kappa_p^p \, \mathrm{d}s &= \int_0^L \eta \langle \varphi_p, \nabla_{T_p}^2 Z \rangle \, \mathrm{d}s \\
&= \left[ \eta \langle \varphi_p, \nabla_{T_p} Z_p \rangle \right]_0^L - \int_0^L \eta' \langle \varphi_p, \nabla_{T_p} Z_p \rangle + \eta \langle \nabla_{T_p} \varphi_p, \nabla_{T_p} Z_p \rangle \, \mathrm{d}s \\
&=  - \int_0^L \eta \langle \nabla_{T_p} \varphi_p, \nabla_{T_p} Z_p \rangle \, \mathrm{d}s \\
&= -\left[ \eta \langle \nabla_{T_p} \varphi_p, Z_p \rangle \right]_0^L + \int_0^L \eta' \langle \nabla_{T_p} \varphi_p, Z_p \rangle + \eta \langle \nabla_{T_p}^2 \varphi_p, Z_p \rangle \, \mathrm{d}s \\
&= \int_0^L \eta' \langle \nabla_{T_p} \varphi_p, Z_p \rangle \, \mathrm{d}s + \int_0^L \eta \langle \nabla_{T_p}^2 \varphi_p, Z_p \rangle \, \mathrm{d}s \\
&= \left[ \eta' \langle \varphi_p, Z_p \rangle \right]_0^L - \int_0^L  \eta'' \langle \varphi_p, Z_p \rangle + \eta' \langle \varphi_p, \nabla_{T_p} Z_p \rangle \, \mathrm{d}s + \int_0^L \eta \langle \nabla_{T_p}^2 \varphi_p, Z_p \rangle \, \mathrm{d}s \\
&= - \int_0^L  \eta'' \langle \varphi_p, Z_p \rangle \, \mathrm{d}s + \int_0^L \eta \langle \nabla_{T_p}^2 \varphi_p, Z_p \rangle \, \mathrm{d}s,
\end{align*}
having used integration by parts as well as the facts that $Z$ vanishes at the boundary and $\nabla_{T_p} Z_p = T_p$ is orthogonal to $\varphi_p$. \\
Substituting in for $\nabla_{T_p}^2 \varphi_p$ using the Euler-Lagrange equation \eqref{eq:ELnorm} then integrating by parts turns this expression into
\begin{align*}
K_p^{1-p} \int_0^L \eta \kappa_p^p \, \mathrm{d}s = &- \int_0^L  \eta'' \langle \varphi_p, Z_p \rangle \, \mathrm{d}s - \int_0^L \eta \langle R(\varphi_p,T_p)T_p, Z_p \rangle \, \mathrm{d}s - \sigma\int_0^L \eta d_p \langle \nu_p, Z_p \rangle \, \mathrm{d}s \\
&+ L \lambda_p \int_0^L \eta d_p \langle \nabla_{T_p} T_p, Z_p \rangle \, \mathrm{d}s + \frac{2p-1}{p-1} \int_0^L \eta \langle \varphi_p, \nabla_{T_p} T_p \rangle  \, \mathrm{d}s \\
&+ \frac{2p-1}{p-1} \int_0^L \eta' \langle \varphi_p, \nabla_{T_p} T_p \rangle \langle T_p, Z_p \rangle  \, \mathrm{d}s + \frac{\sigma}{2} \int_0^L \eta (d_p^2)' \langle T_p, Z_p \rangle \, \mathrm{d}s  \\
&+ \frac{\sigma}{2}\int_0^L \eta d_p^2 \langle \nabla_{T_p} T_p, Z_p \rangle \, \mathrm{d}s,
\end{align*}
or equivalently after some manipulation
\begin{align*}
\frac{p}{p-1} K_p^{1-p} \int_0^L \eta \kappa_p^p \, \mathrm{d}s = &\, \int_0^L  \eta'' \langle \varphi_p, Z_p \rangle \, \mathrm{d}s + \int_0^L \eta \langle R(\varphi_p,T_p)T_p, Z_p \rangle \, \mathrm{d}s + \sigma\int_0^L \eta d_p \langle \nu_p, Z_p \rangle \, \mathrm{d}s \\
&- L\lambda_p \int_0^L \eta d_p \langle \nabla_{T_p} T_p, Z_p \rangle \, \mathrm{d}s - \frac{2p-1}{p-1} \int_0^L \eta' \langle \varphi_p, \nabla_{T_p} T_p \rangle \langle T_p, Z_p \rangle  \, \mathrm{d}s \\
&- \frac{\sigma}{2}\int_0^L \eta (d_p^2)' \langle T_p, Z_p \rangle \, \mathrm{d}s - \frac{\sigma}{2}\int_0^L \eta d_p^2 \langle \nabla_{T_p} T_p, Z_p \rangle \, \mathrm{d}s \\
= &\, I_1 + I_2 + I_3 + I_4 + I_5 + I_6 +I_7.
\end{align*}
Now, suppose the limiting $\varphi$ is identically zero, so that $\varphi_p$ converges uniformly to 0 on compact subintervals of $(0,L)$. Set $\eta$ so that the supports of $\eta'$ and of $\eta''$ are contained in such a subinterval, denoted by $[a,b]$. Using the fact that $\vert Z_p \vert \, \leq \delta$ on $\text{supp} \eta$ by construction, along with the same techniques we have used to bound integrals earlier in this paper, we obtain the inequalities
\begin{align*}
I_1, I_5 &\lesssim C_\eta \delta \varepsilon_p , \\
I_2, I_3, I_4, I_6, I_7 &\lesssim \delta
\end{align*} 
follow
, where $\varepsilon_p$ represents a term that decays to 0 as $p \rightarrow \infty$.
The constant $C_\eta$ denotes the quantity $\max( \norm{\eta'}_{L^\infty}, \norm{\eta''}_{L^\infty})$.
\noindent The inequality
\begin{equation}\label{eq:concentrationineq}
K_p^{1-p} \int_0^L \eta \kappa_p^p \, \mathrm{d}s \leq C \left( \delta + C_\eta \delta \varepsilon_p  \right)
\end{equation}
follows, where $C$ is some constant independent of both $p$ and $\eta$. \\

\noindent On the other hand, the uniform convergence of $\varphi_p$ to 0 on $[a,b]$ means that for large $p$, the inequality
\begin{equation*}
\norm{\varphi_p}_{L^{p'}([a,b])} \leq \epsilon \leq \frac{1}{2} \norm{\varphi_p}_{L^{p'}(0,L)}
\end{equation*}
holds. That is, at least half of the ``$p'$ mass'' of $\varphi_p$ is concentrated at the tails $(0,a)$ and $(b,L)$ of $(0,L)$. It is then true that
\[
\norm{\eta^{1/p'} \varphi_p}_{L^{p'}(0,L)} \geq \norm{\eta^{1/p'} \varphi_p}_{L^{p'}((0,L) \backslash [a,b])} = \norm{\varphi_p}_{L^{p'}((0,L) \backslash [a,b])} \geq \frac{1}{2}  \norm{\varphi_p}_{L^{p'}(0,L)} = \frac{L^{1/p'}}{2},
\]
having used the definition of $\varphi_p$ to compute its  $L^{p'}$ norm. Therefore
\[
K_p^{-p} \int_0^L \eta \kappa_p^p \, \mathrm{d}s   \geq \frac{L^{1/p'}}{2},
\]
which means
\[
K_p^{1-p} \int_0^L \eta \kappa_p^p \, \mathrm{d}s   \geq \frac{K_1 \min(1,L)}{2}.
\]
Choosing $\delta$ sufficiently small and $p$ large enough (to control the $C_\eta \delta \varepsilon_p$ term, since $C_\eta \delta$ may grow large when $\delta \rightarrow 0$) gives a contradiction to inequality \eqref{eq:concentrationineq}. \\

\noindent The only step left is to prove statement 2) of Theorem \ref{thm:limitingeqns}.\\
Where $\varphi \neq 0$, it is immediate from equation \eqref{eq:EL2} that $\vert \nabla_T T \vert = K$. It remains therefore to show that $\nabla_T T = 0$ almost everywhere on the set $\Sigma = \varphi^{-1}(\{ 0 \} )$.
Consider two different cases depending on the value of $\lambda$. \\
First, assume that $\lambda \neq 0$. By Rademacher's theorem and the fact that $W^{1,\infty} = C^{0,1}$, we see that the second derivative $\nabla_T^2 \varphi$ exists in the classical sense almost everywhere in $(0,L)$ and so it makes sense to interpret $\nabla_T^2 \varphi$ pointwise. Substituting \eqref{eq:EL2} into \eqref{eq:EL1}, we see that at almost every given point $s_1 \in \Sigma$ the equation
\begin{equation}\label{eq:measzeroproof}
\nabla_{T}^2 \varphi = L \lambda \nabla_{T} T - 2K \vert \varphi \vert' T
\end{equation}
holds. If $\nabla_{T}^2 \varphi(s_1) = 0$, taking the inner product with $\nabla_T T$ in \eqref{eq:measzeroproof} shows that $\nabla_T T(s_1) = 0$. If instead $\nabla_{T}^2 \varphi(s_1) \neq 0$,
it follows that $s_1$ is an isolated point of $\Sigma$ and can therefore be ignored for the purposes of this proof. \\
Now, assume instead that $\lambda = 0$. Substituting \eqref{eq:EL2} into \eqref{eq:EL1} as before gives us the equation
\[
\nabla_{T}^2 \varphi = - R(\varphi, T)T  - 2K \vert \varphi \vert' T - 2K \vert \varphi \vert \nabla_T T,
\]
which implies the inequality
\[
\vert \nabla_{T}^2 \varphi \vert \,\leq C\left( \vert \varphi \vert  + \vert \varphi \vert' \right)
\]
after using some of the bounds established earlier. Writing out the expression for $\nabla_T^2 \varphi$ in local co-ordinates
\[
\nabla_{T}^2 \varphi = \varphi^i{}'' \vecpar{i} + 2 \varphi^i{}'T^j\Gamma_{ji}^k\vecpar{k} + \varphi^i \left( T^j \Gamma_{ji}^k \right)'\vecpar{k} + \varphi^i T^j \Gamma_{ji}^k T^m \Gamma_{mk}^s \vecpar{s},
\]
rearranging to get an expression for $\varphi''$ then using the triangle inequality and equivalence of norms implies the inequality
\[
\vert \varphi'' \vert\, \leq C\left( \vert \varphi \vert  + \vert \varphi' \vert \right).
\]
where the norms here are taken to be the Euclidean norms of the vector $\varphi$ and its derivatives, viewed as vectors in $\mathbb{R}^n$ via local co-ordinates. From this we may obtain the (Euclidean) inequality
\[
\vert \phi \vert' \,\leq \tilde{C} \vert \phi \vert
\]
for $\phi = (\varphi, \varphi')$. By Gr\"{o}nwall's inequality and the fact that $\varphi$ is not identically zero it follows that the only zeroes of $\varphi$ must be isolated.
\end{proof}

\noindent All the results we have established so far apply to $\infty$-elastica, and therefore to minimisers of $\mathcal{K}_\infty$. However, we have not yet shown that such minimisers actually exist; the following remark rectifies this and ensures that the analysis of this paper is meaningful.
\begin{rem}\label{rem:minexist}
Under the conditions of the problem described in Section \ref{sec:intro}, there exists a minimiser of $\mathcal{K}_\infty$.
\end{rem}
\noindent Since the proof of this remark is so similar to the proofs already contained in this paper, we do not go into full detail here but rather give a sketch. Consider a sequence of curves $(\gamma_p)$ minimising the $\mathcal{K}_p$ functional (that is, the $\mathcal{J}_p^\sigma$ functional with $\sigma = 0$, i.e. no penalisation term). Using arguments similar to those contained in Sections \ref{sec:approx}--\ref{sec:eqns}, we obtain the (subsequential) convergence to a limiting curve $\gamma_\infty$ satisfying conditions \eqref{eq:bcs}--\eqref{eq:length} with $\mathcal{K}_\infty \left[ \gamma_\infty \right] = \lim_p \mathcal{K}_p \left[ \gamma_p \right]$. A straightforward manipulation of inequalities shows that there cannot exist any admissible curves such that the $L^\infty$ norm of their curvature is less than that of $\gamma_\infty$, or else there would exist $p$ such that $\gamma_p$ does not minimise $\mathcal{K}_p$, a contradiction. This shows that $\gamma_\infty$ is in fact a minimiser of $\mathcal{K}_\infty$. \\

\noindent In general we cannot expect that minimisers of $\mathcal{K}_\infty$ are unique; indeed, we can generalise the Euclidean constructions from \cite[Section~1]{moser2019structure} to create examples where non-uniqueness is guaranteed. For example, when the boundary conditions \eqref{eq:bcs} are symmetric with respect to some isometry of $M$ yet no admissible curves exist which are invariant under this isometry, there are necessarily multiple minimisers. \\



\noindent \textbf{Acknowledegements.} Ed Gallagher is grateful for being funded by a studentship from the EPSRC, project reference 2446338.

\newpage
\bibliographystyle{unsrt}
\bibliography{bibliography}

\end{document}